\documentclass{amsart}
\usepackage{mathrsfs}
\usepackage{mathrsfs}
\usepackage{amsfonts}
\usepackage{cases}
\usepackage{latexsym}
\usepackage{amsmath}
\usepackage[all]{xy}
\usepackage{stmaryrd}
\usepackage{amsfonts}
\usepackage{amsmath,amssymb,amscd,bbm,amsthm,mathrsfs,dsfont}

\usepackage[numbers,sort&compress]{natbib}
\usepackage[pagebackref]{hyperref}
\usepackage{hypernat}
 
\usepackage{color, hyperref}
\definecolor{blue}{rgb}{1.0,0.0,0.9}
\hypersetup{colorlinks, breaklinks,
            linkcolor=blue,urlcolor=blue,
            anchorcolor=blue, citecolor=blue}
\usepackage{fancyhdr}
\usepackage{amsxtra,ifthen}
\usepackage{verbatim}

\usepackage[numbers,sort&compress]{natbib}
\usepackage[pagebackref]{hyperref}
\usepackage{hypernat}

\theoremstyle{plain}
\newtheorem{theorem}{Theorem}[section]

\newtheorem{lemma}[theorem]{Lemma}

\newtheorem{question}[theorem]{Question}

\theoremstyle{definition}

\begin{document}

\title[Nonlinear $\ast$-Jordan-Type Derivations on von Neumann Algebras]
{Nonlinear $\ast$-Jordan-Type Derivations on von Neumann Algebras}

\author{Wenhui Lin}

\address{Lin: College of Science, China Agricultural University, Beijing 100083, P. R. China}

\email{whlin@cau.edu.cn}\email{haotianwei@hotmail.com}

\begin{abstract}
Let $\mathcal{H}$ be a complex Hilbert space, $\mathcal{B(H)}$ be the algebra of all 
bounded linear operators on $\mathcal{H}$ and $\mathcal{A} \subseteq \mathcal{B(H)}$ 
be a von Neumann algebra without central summands of type $I_1$.  For arbitrary elements 
$A, B\in \mathcal{A}$, one can define their $\ast$-Jordan product in the sense of 
$A\diamond B = AB+BA^\ast$. Let  $p_n(x_1,x_2,\cdots,x_n)$ be the polynomial 
defined by $n$ indeterminates $x_1, \cdots, x_n$ 
and their $\ast$-Jordan products. In this article, it is shown that a mapping 
$\delta: \mathcal{A} \longrightarrow  \mathcal{B(H)}$ satisfies the condition 
$$
\delta(p_n(A_1, A_2,\cdots, A_n))=\sum_{k=1}^n
p_n(A_1,\cdots, A_{k-1}, \delta(A_k), A_{k+1},\cdots, A_n)
$$
for all $A_1, A_2,\cdots, A_n \in \mathcal{A}$ if and only if $\delta$ is an additive $\ast$-derivation. 
\end{abstract}

\subjclass[2010]{47B47, 46L10}

\keywords{Nonlinear $\ast$-Jordan-type derivation, von Neumann algebra}


\date{\today}

\maketitle

\tableofcontents

\section{Introduction}
\label{xxsec1}

Let $\mathcal{A}$ be an associative $\ast$-algebra over the complex field $\Bbb{C}$. 
For any $A, B\in \mathcal{A}$, one can denote a ``new product" of $A$ and $B$ by 
$A\diamond B=AB+BA^\ast$, and this new product $\diamond$ is usually said to be \textit{$\ast$-Jordan product}. 
Such kind of product based on Jordan bracket naturally 
appears in relation with the so-called Jordan $\ast$-derivations and plays an important 
role in the problem of representability of quadratic functionals by sesqui-linear 
functionals on left-modules over $\ast$-algebras (see \cite{Molnar2, MolnarSemrl, Semrl1}). 
The product is workable for us to characterize ideals, see  \cite{BresarFosner, Molnar1, Molnar3}. 
Especial attention has been paid to understanding mappings which preserve the product 
$AB+BA^\ast$ between $\ast$-algebras, see \cite{DaiLu, HuoZhengLiu, LiLu, LiLuFang1, ZhaoLi2}.

The question of to what extent the multiplicative structure of an algebra determines its 
additive structure has been considered by many researchers  over the past decades. 
In particular, they have investigated under which conditions bijective mappings 
between algebras preserving the multiplicative structure necessarily preserve the 
additive structure as well. The most fundamental result in this direction is due 
to W. S. Martindale III \cite{Martindale} who  proved that every bijective 
multiplicative mapping from a prime ring containing a nontrivial idempotent 
onto an arbitrary ring is necessarily additive. Later, a number of authors 
considered the Jordan-type product or Lie-type product and proved that, 
on certain associative algebras or rings, bijective mappings which preserve 
any of those products are automatically additive, see 
\cite{BaiDu, BresarFosner, CuiLi, DaiLu, HuoZhengLiu, LiLu, LiLuFang1, Molnar4, Molnar5, ZhaoLi2}.

An additive mapping $\delta: \mathcal{A} \longrightarrow \mathcal{A}$ is called an
\textit{additive derivation} if $\delta(AB)=\delta(A)B+A\delta (B)$ for all  $A,B \in \mathcal{A}$. 
Furthermore, $\delta$ is said to be an \textit{additive $\ast$-derivation} provided that $\delta$ 
is an additive derivation and satisfies $\delta(A^{\ast})=\delta(A)^\ast$ for all $A\in \mathcal{A}$.
Let $\delta: \mathcal{A}\longrightarrow \mathcal{A}$ be a mapping 
(without the additivity assumption). We say that $\delta$ is a 
\textit{nonlinear $\ast$-Jordan derivation} if
$$
\delta(A \diamond B)=\delta(A) \diamond B+A \diamond \delta(B),
$$
holds true for all $A,B\in \mathcal{A}$. Similarly, a mapping $\delta:
\mathcal{A}\longrightarrow \mathcal{A}$ is called a \textit{nonlinear $\ast$-Jordan triple derivation} if it satisfies the condition
$$
\delta(A \diamond B \diamond C)=\delta(A) \diamond B \diamond C+A \diamond \delta(B)\diamond C+A \diamond B \diamond \delta(C)
$$
for all $A, B, C\in \mathcal{A}$, where $A \diamond B \diamond C=(A \diamond B) \diamond C$ . 
We should be aware that $\diamond$ is not necessarily associative.  

Given the consideration of $\ast$-Jordan derivations and $\ast$-Jordan triple derivations, we can 
further develop them in one natural way. Suppose that $n\geq 2$ is a fixed positive
integer. Let us see a sequence of polynomials with $\ast$
$$
\begin{aligned}
p_1(x_1)&=x_1,\\
p_2(x_1,x_2)&=x_1 \diamond x_2=x_1x_2+x_2x_1^\ast,\\
p_3(x_1,x_2,x_3)&=p_2(x_1,x_2) \diamond x_3=(x_1 \diamond x_2) \diamond x_3,\\
p_4(x_1,x_2,x_3,x_4)&=p_3(x_1,x_2,x_3) \diamond x_4=((x_1 \diamond x_2) \diamond x_3) \diamond x_4,\\
\cdots &\cdots,\\
p_n(x_1,x_2,\cdots,x_n)&=p_{n-1}(x_1,x_2,\cdots,x_{n-1}) \diamond x_n\\
&=\underbrace{(\cdots ((}_{n-2}x_1 \diamond x_2)\diamond x_3)\diamond \cdots \diamond x_{n-1}) \diamond x_n.
\end{aligned}
$$
Accordingly, a \textit{nonlinear $\ast$-Jordan $n$-derivation} is a mapping $\delta:
\mathcal{A} \longrightarrow \mathcal{A}$ satisfying the condition
$$
\delta(p_n(A_1, A_2,\cdots, A_n))=\sum_{k=1}^n
p_n(A_1,\cdots, A_{k-1}, \delta(A_k), A_{k+1},\cdots, A_n)
$$
for all $A_1,A_2,\cdots, A_n\in \mathcal{A}$. This notion makes the best use of the definition of Lie-type derivations 
and that of $\ast$-Lie-type derivations,  see \cite{FosnerWeiXiao, Lin1, Lin2}. 
By the definition, it is clear that every $\ast$-Jordan derivation is a 
$\ast$-Jordan 2-derivation and each $\ast$-Jordan triple derivation is a $\ast$-Jordan 3-derivation. 
One can easily check that each nonlinear $\ast$-Jordan derivation on $\mathcal{A}$ is 
a nonlinear $\ast$-Jordan triple derivation. But, we don't know whether the converse statement is true.
$\ast$-Jordan 2-derivations, $\ast$-Jordan 3-derivations and $\ast$-Jordan $n$-derivations 
are collectively referred to as \textit{$\ast$-Jordan-type derivations}. 
$\ast$-Jordan-type derivations on operator algebras have been studied by 
several authors. Let $\mathcal{H}$ be a complex Hilbert space and $\mathcal{B(H)}$ be the algebra of all 
bounded linear operators on $\mathcal{H}$.  Li et al in \cite{LiLuFang2} showed that if $\mathcal{A}\subseteq \mathcal{B(H)}$ is a von Neumann 
algebra without central summands of type $I_1$, then $\delta: \mathcal{A}\longrightarrow \mathcal{B(H)}$ is a 
nonlinear $\ast$-Jordan derivation if and only if $\delta$ is an additive $\ast$-derivation. 
More recently, this result is extended to the case of nonlinear $\ast$-Jordan triple derivations by Zhao and Li \cite{ZhaoLi1}. 
Taghavi et al \cite{TaghaviRohiDarvish} and Zhang \cite{Zhang} independently investigate 
$\ast$-Jordan derivations on factor von Neumann algebras, respectively. It turns out that 
each nonlinear $\ast$-Jordan derivation on a factor von Neumann algebra is an additive $\ast$-derivation.

Inspired by the afore-mentioned works, we will concentrate on giving a description of nonlinear 
$\ast$-Jordan-type derivations on von Neumann algebras. The organization of this paper is as follows. 
We recall and collect some indispensable facts with respect to von Neumann algebras in the 
second section \ref{xxsec2}. The third Section \ref{xxsec3} is devoted to our main result 
Theorem \ref{xxsec3.1} and its proof. The main theorem states that every nonlinear 
$\ast$-Jordan-type derivation on a von Neumann algebra without central summands of type $I_1$
is an additive $\ast$-derivation. Similar statements are also given for factor von Neumann algebras 
and standard operator algebras without proofs. Some potential topics for the future research are presented in the last Section \ref{xxsec4}.

\section{Preliminaries}
\label{xxsec2}

Throughout this paper, $\mathcal{H}$ denotes a complex Hilbert space and $\mathcal{B}(\mathcal{H})$ 
is the algebra of all bounded linear operators on $\mathcal{H}$. A \textit{von Neumann algebra} 
$\mathcal{A}$ is weakly closed, self-adjoint algebra of operators on $\mathcal{H}$ containing 
the identity operator $I$. The set $\mathcal{Z}(\mathcal{A}) =\{ S \in \mathcal{A}: ST=TS \ \ \text{for all}\ \ T\in \mathcal{A}\}$ 
is called the \textit{centre} of $\mathcal{A}$.  
A projection $P$ is called a \textit{central abelian projection} if $P \in \mathcal{Z}(\mathcal{A})$ 
and $P\mathcal{A}P$ is abelian. Recall that the \textit{central carrier} of $A$, denoted by $\overline{A}$, is 
the smallest central projection $P$ satisfying the condition $PA=A$. It is straightforward to 
check that the central carrier of $A$ is the projection onto the closed subspace spanned by 
$\{ BA(x):B \in \mathcal{A}, x\in \mathcal{H} \}$. If $A$ is self-adjoint, then the \textit{core} 
of $A$, denoted by $\underline{A}$, is $\text{sup}\{ S\in \mathcal{Z}(\mathcal{A}): S=S^{\ast}, S \leq A\}$. If $P$
is a projection, it is clear that $\underline{P}$ is the largest central projection $Q$ with $Q\leq P$. 
A projection $P$ is said to be \textit{core-free} if $\underline{P}=0$. It is not difficult to see 
that $\underline{P}=0$ if and only if $\overline{I-P}=I$.

To round off the proof of our main theorem, we need to give some necessary lemmas.

 \begin{lemma}\label{xxsec2.1}{\rm \cite[Lemma 4]{Miers}}
Let $\mathcal{A}$ be a von Neumann without central summands of type $I_1$.  
Then each nonzero central projection in $\mathcal{A}$ is the central carrier of 
a core-free projection in $\mathcal{A}$.
\end{lemma}

\begin{lemma} \label{xxsec2.2}{\rm \cite[Lemma 2.2]{LiLuFang2}}
Let $\mathcal{A}$ be a von Neumann algebra on a Hilbert space $\mathcal{H}$. 
Let  $A \in \mathcal{B}(\mathcal{H})$ and $P \in \mathcal{A}$ is a projection with 
$\overline{P}=I$. If $ABP=0$ for all $B\in \mathcal{A}$, then $A=0$.

\end{lemma}

\begin{lemma}\label{xxsec2.3}
Let $\mathcal{A}$ be a von Neumann algebra without central summands of type $I_1$. 
For any $A\in \mathcal{A}$ and for any positive integer $n\geq 2$, we have
$$
p_n\left (A, \frac{1}{2}I, \cdots, \frac{1}{2}I \right)= \frac{1}{2}\left(A+A^{\ast} \right) 
$$
and 
$$
p_n\left (\frac{1}{2}I, \cdots, \frac{1}{2}I,A \right)= A. 
$$
 \end{lemma}

\begin{proof}
By a recursive calculation, we know that   
$$
\begin{aligned}
p_n\left(A, \frac{1}{2}I, \cdots, \frac{1}{2}I \right )&=p_{n-1}\left(\frac{1}{2}(A+A^{\ast}), \frac{1}{2}I, \cdots, \frac{1}{2}I\right )\\
&=p_{n-2}\left(\frac{1}{2} \left( A+A^{\ast} \right), \frac{1}{2}I, \cdots, \frac{1}{2}I \right )\\
&=p_{n-3}\left(\frac{1}{2} \left( A+A^{\ast} \right), \frac{1}{2}I, \cdots, \frac{1}{2}I \right )\\
&=\cdots\\
&=\frac{1}{2}\left ( A+A^{\ast} \right ).  
\end{aligned}
$$
Similarly, we also have
$$
\begin{aligned}
p_n\left (\frac{1}{2}I, \cdots, \frac{1}{2}I,A \right)&=p_{n-1}\left (\frac{1}{2}I, \cdots, \frac{1}{2}I,A \right)\\
&=p_{n-2}\left (\frac{1}{2}I, \cdots, \frac{1}{2}I,A \right)\\
&=\cdots\\
 &=A. 
\end{aligned}
$$
\end{proof}

Let $\mathcal{A}$ be a von Neumann algebra without central summands of type $I_1$. 
By Lemma \ref{xxsec2.1}, we know that there exists a nonzero central projection 
$P$ such that $\underline{P}=0$ and $\overline{P}=I$. For the convenience of discussion, 
let us set $P_1=P$, $P_2=I-P$. We write $\mathcal{A}_{ij}=P_i\mathcal{A}P_j$. 
Thus one gets $\mathcal{A}=\sum_{i,j=1}^2 \mathcal{A}_{ij}$. We denote the imaginary 
unit by ${\rm i}$.

\begin{lemma}\label{xxsec2.4} 
Let $\mathcal{A}$ be a von Neumann algebra without central summands of type $I_1$. 
For any  $A\in \mathcal{A}$ with $A=\sum_{i,j =1}^2 A_{ij}$ and $A_{ij}\in \mathcal{A}_{ij}$, we have
\begin{enumerate}
\item[(a)] $P_1 \diamond A=0$ implies that $A_{12}=A_{21}=A_{11}=0$. 
\item[(b)] $P_2 \diamond A=0$ implies that $A_{12}=A_{21}=A_{22}=0$. 
\item[(c)] $(P_2-P_1)\diamond A=0$ implies that $A_{11}=A_{22}=0$.
\end{enumerate}
\end{lemma}

\begin{proof}
Let us first prove the assertion (a).
We have
$$
\begin{aligned}
0&=P_1 \diamond A=P_1A+AP_1^{\ast}=P_1A+AP_1\\
&=A_{12}+A_{21}+2A_{11},
\end{aligned}
$$
which leads to $A_{12}=A_{21}=A_{11}=0$.

The other two assertions can be achieved by an analogous manner. 
\end{proof}

\section{Main Theorem and Its Proof}
\label{xxsec3}

We are in a position to give the main theorem of this article which can be stated 
as follows. 

 \begin{theorem} \label{xxsec3.1}
Let $\mathcal{A}$ be a von Neumann algebra without central summands of 
type $I_1$. Then a mapping $\delta: \mathcal{A} \longrightarrow  \mathcal{B(H)}$ satisfies the rule
$$
\delta(p_n(A_1, A_2,\cdots, A_n))=\sum_{k=1}^n
p_n(A_1,\cdots, A_{k-1}, \delta(A_k), A_{k+1},\cdots, A_n)
$$
for all $A_1, A_2,\cdots, A_n \in \mathcal{A}$ if and only if $\delta$ is an additive $\ast$-derivation.
 \end{theorem}

\begin{proof}
The proof of this theorem can be realized via a series of claims.

\vspace{2mm}

\noindent {\bf Claim 1.} $\delta(0)=0$.
$$
\delta(0)=\delta(p_n(0,0,\cdots,0))=\sum_{k=1}^{n}\overset{\ \ \ \  k}{p_n(0,\cdots,\delta(0),\cdots,0)}=0.
$$

\noindent {\bf Claim 2.}  For any $A \in \mathcal{A}$, we have $\delta \left( \frac{1}{2}I \right) \diamond A=0$.
\vspace{2mm}

Note that  the fact $\frac{1}{2}I=p_n\left(\frac{1}{2}I,\cdots,\frac{1}{2}I \right)$.  By Lemma \ref{xxsec2.3} it follows that 
$$
\begin{aligned}
\delta \left( \frac{1}{2}I \right)=&\delta\left (p_n\left(\frac{1}{2}I,\cdots,\frac{1}{2}I \right) \right)\\
=&\sum_{k=1}^{n}\overset{\ \ \ \  k}{p_n\left(\frac{1}{2}I,\cdots,\frac{1}{2}I,\delta \left( \frac{1}{2}I \right),\frac{1}{2}I,\cdots,\frac{1}{2}I \right)}\\
=&\frac{n-1}{2}\left ( \delta \left( \frac{1}{2}I \right) +\delta \left( \frac{1}{2}I \right)^{\ast}\right)+\delta \left( \frac{1}{2}I \right).
\end{aligned}
$$
This gives
$$
\delta \left( \frac{1}{2}I \right) +\delta \left( \frac{1}{2}I \right)^{\ast}=0.  
$$
That is, 
$$
\delta \left( \frac{1}{2}I \right) \diamond \frac{1}{2}I=0. \eqno(3.1)
$$
Using the relation (3.1), we get
$$ 
\begin{aligned}
\delta(A)=&\delta\left (p_n\left(\frac{1}{2}I,\cdots,\frac{1}{2}I,A \right)\right)\\
=&\sum_{k=1}^{n-1}\overset{\ \ k}{p_n\left(\frac{1}{2}I,\cdots,\frac{1}{2}I,\delta \left( \frac{1}{2}I \right),\frac{1}{2}I,\cdots,\frac{1}{2}I,A \right)}+p_n\left(\frac{1}{2}I,\cdots,\frac{1}{2}I,\delta(A) \right)\\
=&\delta \left( \frac{1}{2}I \right)A +A\delta \left( \frac{1}{2}I \right)^{\ast}+\delta(A).
\end{aligned}
$$
for all $A \in \mathcal{A}$. Thus we obtain
$$ 
\delta \left( \frac{1}{2}I \right)A +A\delta \left( \frac{1}{2}I \right)^{\ast}=0.
$$
That is, 
$$
\delta \left( \frac{1}{2}I \right) \diamond A=0.
$$

\noindent{\bf Claim 3.} For any $A_{ll}\in \mathcal{A}_{ll}, B_{ij}\in \mathcal{A}_{ij}\, (i,j, l=1,2, i \neq j)$, we have
$$
\delta(A_{ll}+B_{ij})=\delta(A_{ll})+\delta(B_{ij}).
$$
We only need to prove the case of $i=l=1,j=2$, and the proofs of the other cases are rather similar 
and are omitted here. Let us write
$$
M=\delta(A_{11}+B_{12})-\delta(A_{11})-\delta(B_{12}).
$$
It is sufficient for us to show that $M=0$. Since
$$
p_n\left ( \frac{1}{2}I,\cdots, \frac{1}{2}I,P_2,A_{11} \right)=0
$$
and
$$
p_n\left (\frac{1}{2}I,\cdots, \frac{1}{2}I, P_2,A_{11}+ B_{12} \right)= p_n\left (\frac{1}{2}I,\cdots, \frac{1}{2}I,  P_2,B_{12} \right),
$$
we by {\bf Claim 2} have
$$
\begin{aligned}
&p_n\left ( \frac{1}{2}I,\cdots, \frac{1}{2}I, \delta\left(P_2\right),A_{11}+ B_{12}\right)+ p_n\left ( \frac{1}{2}I,\cdots, \frac{1}{2}I,P_2, \delta\left(A_{11}+ B_{12}\right) \right)\\
=&\delta\left( p_n\left ( \frac{1}{2}I,\cdots, \frac{1}{2}I,P_2, A_{11}+ B_{12} \right)\right)\\
=&\delta\left( p_n\left ( \frac{1}{2}I,\cdots, \frac{1}{2}I, P_2,A_{11} \right)\right)+\delta\left( p_n\left ( \frac{1}{2}I,\cdots, \frac{1}{2}I,P_2, B_{12} \right)\right)\\
=&p_n\left ( \frac{1}{2}I,\cdots, \frac{1}{2}I,\delta\left(P_2\right), A_{11}\right)+ p_n\left ( \frac{1}{2}I,\cdots, \frac{1}{2}I,P_2, \delta\left(A_{11}\right) \right)\\
&+p_n\left ( \frac{1}{2}I,\cdots, \frac{1}{2}I, \delta\left(P_2\right),B_{12} \right)+p_n\left ( \frac{1}{2}I,\cdots, \frac{1}{2}I, P_2,\delta\left(B_{12}\right) \right).
\end{aligned}
$$ 
We therefore get 
$$
p_n\left (\frac{1}{2}I,\cdots, \frac{1}{2}I,P_2,M \right)=0.
$$
In light of Lemma \ref{xxsec2.3}, we obtain
$$
p_2\left ( P_2, M \right)=P_2 \diamond M=0.
$$
It follows from Lemma \ref{xxsec2.4} that 
$$
M_{12}=M_{21}=M_{22}=0.
$$
Notice that
$$
p_n\left ( \frac{1}{2}I,\cdots, \frac{1}{2}I,P_2-P_1, B_{12} \right)= ( P_2-P_1)\diamond B_{12}=0
$$
and
$$
p_n\left ( \frac{1}{2}I,\cdots, \frac{1}{2}I,P_2-P_1, A_{11}+ B_{12} \right)= p_n\left ( \frac{1}{2}I,\cdots, \frac{1}{2}I, P_2-P_1,A_{11}\right).
$$
By {\bf Claim 2}, we observe that 
$$
\begin{aligned}
&p_n\left ( \frac{1}{2}I,\cdots, \frac{1}{2}I, \delta(P_2-P_1),A_{11}+ B_{12} \right)\\
&+ p_n\left (\frac{1}{2}I,\cdots, \frac{1}{2}I, P_2-P_1, \delta\left(A_{11}+ B_{12}\right) \right)\\
=&\delta\left( p_n\left ( \frac{1}{2}I,\cdots, \frac{1}{2}I, P_2-P_1,A_{11}+ B_{12}\right)\right)\\
=&\delta\left( p_n\left ( \frac{1}{2}I,\cdots, \frac{1}{2}I, P_2-P_1,B_{12} \right)\right)+\delta\left( p_n\left (\frac{1}{2}I,\cdots, \frac{1}{2}I,  P_2-P_1,A_{11} \right)\right)\\
=&p_n\left ( \frac{1}{2}I,\cdots, \frac{1}{2}I,\delta(P_2-P_1), B_{12} \right)+p_n\left (\frac{1}{2}I,\cdots, \frac{1}{2}I, P_2-P_1, \delta\left(B_{12}\right) \right)\\
&+p_n\left (\frac{1}{2}I,\cdots, \frac{1}{2}I, \delta(P_2-P_1), A_{11}  \right)+ p_n\left ( \frac{1}{2}I,\cdots, \frac{1}{2}I, P_2-P_1, \delta\left(A_{11} \right) \right).
\end{aligned}
$$ 
Thus we arrive at
$$
p_n\left (\frac{1}{2}I,\cdots, \frac{1}{2}I, P_2-P_1, M \right)=0.
$$
Taking into account Lemma \ref{xxsec2.3}, we get
$$
p_2\left (  P_2-P_1, M \right)=(P_2-P_1)\diamond M=0.
$$
Applying Lemma \ref{xxsec2.4} yields that
$$
M_{11}=0.
$$
We therefore have $M=0$. That is, 
$$
\delta(A_{11}+B_{12})=\delta(A_{11})+\delta(B_{12}).
$$
The other cases can be verified by an analogous manner.

\vspace{2mm}

\noindent{\bf Claim 4.} For any $B_{12}\in \mathcal{A}_{12}, C_{21}\in \mathcal{A}_{21}$, we have
$$
\delta(B_{12}+C_{21})=\delta(B_{12})+\delta(C_{21}).
$$
We only need to show that
$$
M=\delta(B_{12}+C_{21})-\delta(B_{12})-\delta(C_{21})=0.
$$
Since
$$
p_n\left (\frac{1}{2}I,\cdots, \frac{1}{2}I,P_2-P_1, B_{12} \right)= (P_2-P_1)\diamond B_{12}=0
$$
and
$$
p_n\left ( \frac{1}{2}I,\cdots, \frac{1}{2}I,P_2-P_1,C_{21} \right)= p_n\left ( \frac{1}{2}I,\cdots, \frac{1}{2}I,P_2-P_1,B_{12}+C_{21} \right)=0,
$$
we obtain
$$
\begin{aligned}
&p_n\left ( \frac{1}{2}I,\cdots, \frac{1}{2}I,\delta(P_2-P_1), B_{12}+C_{21} \right)\\
&+ p_n\left ( \frac{1}{2}I,\cdots, \frac{1}{2}I,P_2-P_1, \delta\left(B_{12}+C_{21}\right) \right)\\
=&\delta\left( p_n\left ( \frac{1}{2}I,\cdots, \frac{1}{2}I,P_2-P_1, B_{12}+C_{21} \right)\right)\\
=&\delta\left( p_n\left ( \frac{1}{2}I,\cdots, \frac{1}{2}I, P_2-P_1,B_{12} \right)\right)+\delta\left( p_n\left ( \frac{1}{2}I,\cdots, \frac{1}{2}I,P_2-P_1, C_{21} \right)\right)\\
=&p_n\left ( \frac{1}{2}I,\cdots, \frac{1}{2}I, \delta(P_2-P_1),B_{12} \right)+p_n\left ( \frac{1}{2}I,\cdots, \frac{1}{2}I,P_2-P_1,\delta\left(B_{12}\right) \right)\\
&+p_n\left ( \frac{1}{2}I,\cdots, \frac{1}{2}I, \delta(P_2-P_1), C_{21} \right)+ p_n\left ( \frac{1}{2}I,\cdots, \frac{1}{2}I,P_2-P_1, \delta\left(C_{21}\right) \right).
\end{aligned}
$$ 
Hence, we have
$$
p_n\left ( \frac{1}{2}I,\cdots, \frac{1}{2}I, P_2-P_1,M \right)=0.
$$
Applying Lemma \ref{xxsec2.3} gives
$$
p_2\left ( P_2-P_1, M \right)=(P_2-P_1)\diamond M=0.
$$
By Lemma \ref{xxsec2.4}, we know that 
$$
M_{11}=M_{22}=0.
$$
Note that the facts
$$
p_n\left (\frac{1}{2}I,\cdots, \frac{1}{2}I,B_{12},P_1 \right)=0
$$
and
$$
p_n\left ( \frac{1}{2}I,\cdots, \frac{1}{2}I, B_{12}+C_{21},P_1\right)=p_n\left ( \frac{1}{2}I,\cdots, \frac{1}{2}I,C_{21},P_1 \right).
$$
Using similar computations as the above, we get
$$
p_n\left (\frac{1}{2}I,\cdots, \frac{1}{2}I,M,P_1 \right)=0.
$$
In view of Lemma \ref{xxsec2.3} and the fact $M_{11}=M_{22}=0$, one can see that
$$
M_{21}=0.
$$
On the other hand, we should remark that 
$$
p_n\left (\frac{1}{2}I,\cdots, \frac{1}{2}I, B_{12}+C_{21},P_2\right)=p_n\left (\frac{1}{2}I,\cdots, \frac{1}{2}I,B_{12},P_2 \right)
$$
and
$$
p_n\left ( \frac{1}{2}I,\cdots, \frac{1}{2}I,C_{21},P_2 \right)=0.
$$
Using similar arguments as the above, one can get $M_{12}=0$. 
\vspace{2mm}

\noindent{\bf Claim 5.} For all $A_{11}\in \mathcal{A}_{11}, D_{22}\in \mathcal{A}_{22}$, we have
$$
\delta(A_{11}+D_{22})=\delta(A_{11})+\delta(D_{22}).
$$
It is sufficient to prove that 
$$
M=\delta(A_{11}+D_{22})-\delta(A_{11})-\delta(D_{22})=0.
$$
Since
$$
p_n\left ( \frac{1}{2}I,\cdots, \frac{1}{2}I,P_2, A_{11} \right)=0
$$
and
$$
p_n\left ( \frac{1}{2}I,\cdots, \frac{1}{2}I,P_2,A_{11}+ D_{22} \right)= p_n\left (\frac{1}{2}I,\cdots, \frac{1}{2}I,P_2, D_{22} \right),
$$
we konw that 
$$
\begin{aligned}
&p_n\left (\frac{1}{2}I,\cdots, \frac{1}{2}I, \delta\left(P_2\right), A_{11}+ D_{22}\right)+ p_n\left ( \frac{1}{2}I,\cdots, \frac{1}{2}I,P_2, \delta\left(A_{11}+ D_{22}\right) \right)\\
=&\delta\left( p_n\left ( \frac{1}{2}I,\cdots, \frac{1}{2}I, P_2,A_{11}+ D_{22}\right)\right)\\
=&\delta\left( p_n\left ( \frac{1}{2}I,\cdots, \frac{1}{2}I, P_2,D_{22}\right)\right)+\delta\left( p_n\left ( \frac{1}{2}I,\cdots, \frac{1}{2}I,P_2, A_{11} \right)\right)\\
=&p_n\left ( \frac{1}{2}I,\cdots, \frac{1}{2}I,\delta\left(P_2\right),D_{22} \right)+p_n\left (\frac{1}{2}I,\cdots, \frac{1}{2}I, P_2,\delta\left(D_{22}\right) \right)\\
&+p_n\left (\frac{1}{2}I,\cdots, \frac{1}{2}I, \delta\left(P_2\right), A_{11}\right)+ p_n\left (\frac{1}{2}I,\cdots, \frac{1}{2}I, P_2, \delta\left(A_{11}\right) \right).
\end{aligned}
$$ 
Thus we obtain
$$
p_n\left ( \frac{1}{2}I,\cdots, \frac{1}{2}I, P_2,M \right)=0.
$$
By invoking of Lemma \ref{xxsec2.3}, we arrive at
$$
p_2\left ( P_2, M \right)=P_2 \diamond M=0.
$$
It follows from Lemma \ref{xxsec2.4} that
$$
M_{12}=M_{21}=M_{22}=0.
$$
We should remark that 
$$
p_n\left ( \frac{1}{2}I,\cdots, \frac{1}{2}I, P_1,D_{22} \right)= P_1\diamond D_{22}=0
$$
and that
$$
p_n\left (\frac{1}{2}I,\cdots, \frac{1}{2}I, P_1,A_{11}+D_{22} \right)=p_n\left (\frac{1}{2}I,\cdots, \frac{1}{2}I, P_1,A_{11}\right).
$$
Using similar discussions as the above, one can get 
$$
M_{11}=0.
$$
Hence we conclude that $M=0$. That is, 
$$
\delta(A_{11}+D_{22})=\delta(A_{11})+\delta(D_{22}).
$$

\noindent{\bf Claim 6.} For any $A_{11}\in \mathcal{A}_{11}, B_{12}\in \mathcal{A}_{12},C_{21}\in \mathcal{A}_{21}$ and $D_{22}\in \mathcal{A}_{22}$, we have
\begin{enumerate}
\item[(a)] $\delta(A_{11}+B_{12}+C_{21})=\delta(A_{11})+\delta(B_{12})+\delta(C_{21})$,  

\item[(b)]   $\delta(B_{12}+C_{21}+D_{22})=\delta(B_{12})+\delta(C_{21})+\delta(D_{22})$.
\end{enumerate}

Let us first prove the result (a). For convenience, let us set 
$$
M= \delta(A_{11}+B_{12}+C_{21})-\delta(A_{11})-\delta(B_{12})-\delta(C_{21}). 
$$
We shall prove that $M=0$. In view of the facts 
$$
p_n\left (\frac{1}{2}I,\cdots, \frac{1}{2}I , P_2, A_{11} \right)=0
$$
and
$$
p_n\left (\frac{1}{2}I,\cdots, \frac{1}{2}I , P_2, A_{11}+B_{12}+C_{21} \right)=p_n\left (\frac{1}{2}I,\cdots, \frac{1}{2}I , P_2, B_{12}+C_{21} \right),
$$
we by {\bf Claim 4} get
$$
\begin{aligned}
&p_n \left ( \frac{1}{2}I,\cdots, \frac{1}{2}I , \delta\left(P_2\right),A_{11}+B_{12}+C_{21} \right)\\
&+p_n\left (\frac{1}{2}I,\cdots, \frac{1}{2}I , P_2, \delta\left(A_{11}+B_{12}+C_{21} \right)\right)\\
=&\delta \left(p_n \left (\frac{1}{2}I,\cdots, \frac{1}{2}I , P_2, A_{11}+B_{12}+C_{21} \right)\right)\\
=&\delta \left(p_n \left ( \frac{1}{2}I,\cdots, \frac{1}{2}I , P_2,B_{12}+C_{21} \right)\right)+\delta \left(p_n \left ( \frac{1}{2}I,\cdots, \frac{1}{2}I ,P_2, A_{11} \right)\right)\\
=&p_n\left ( \frac{1}{2}I,\cdots, \frac{1}{2}I ,\delta\left(P_2\right), B_{12}  +C_{21}\right)+p_n\left(\frac{1}{2}I,\cdots, \frac{1}{2}I , P_2, \delta\left(B_{12} \right)+\delta\left(C_{21} \right)\right)\\
&+p_n\left ( \frac{1}{2}I,\cdots, \frac{1}{2}I ,\delta\left(P_2\right), A_{11} \right)+p_n\left (\frac{1}{2}I,\cdots, \frac{1}{2}I , P_2, \delta\left(A_{11} \right)\right)
\end{aligned}
$$
By Lemma \ref{xxsec2.3} we know that
$$
p_n\left (\frac{1}{2}I,\cdots, \frac{1}{2}I , P_2, M\right)=P_2\diamond M=0.
$$
It follows from Lemma \ref{xxsec2.4} that
$$
M_{12}=M_{21}=M_{22}=0.
$$
In order to show $M_{11}=0$, we should note that
$$
p_n\left (\frac{1}{2}I,\cdots, \frac{1}{2}I , P_2-P_1, C_{21} \right)=0
$$
and that
$$
p_n\left (\frac{1}{2}I,\cdots, \frac{1}{2}I , P_2-P_1, A_{11}+B_{12}+C_{21} \right)=p_n\left (\frac{1}{2}I,\cdots, \frac{1}{2}I , P_2-P_1, A_{11}+B_{12} \right).
$$
Using {\bf Claim 3}, we see that
$$
\begin{aligned}
&p_n\left ( \frac{1}{2}I,\cdots, \frac{1}{2}I, \delta(P_2-P_1),A_{11}+B_{12}+C_{21} \right)\\
&+ p_n\left (\frac{1}{2}I,\cdots, \frac{1}{2}I, P_2-P_1,\delta\left(A_{11}+B_{12}+C_{21}\right) \right)\\
=&\delta\left( p_n\left (\frac{1}{2}I,\cdots, \frac{1}{2}I,P_2-P_1, A_{11}+B_{12}+C_{21} \right)\right)\\
=&\delta\left( p_n\left (\frac{1}{2}I,\cdots, \frac{1}{2}I, P_2-P_1,A_{11}+B_{12} \right)\right)+\delta\left( p_n\left ( \frac{1}{2}I,\cdots, \frac{1}{2}I, P_2-P_1,C_{21} \right)\right)\\
=&p_n\left ( \frac{1}{2}I,\cdots, \frac{1}{2}I,\delta(P_2-P_1), A_{11}+B_{12} \right)+p_n\left (\frac{1}{2}I,\cdots, \frac{1}{2}I, P_2-P_1, \delta\left(A_{11}\right) +\delta\left(B_{12}\right)\right)\\
&+p_n\left (\frac{1}{2}I,\cdots, \frac{1}{2}I,\delta(P_2-P_1), C_{21} \right)+ p_n\left ( \frac{1}{2}I,\cdots, \frac{1}{2}I, P_2-P_1,\delta\left(C_{21}\right) \right).
\end{aligned}
$$ 
This implies that
 $$
0=p_n\left (\frac{1}{2}I,\cdots, \frac{1}{2}I , P_2-P_1, M\right)=(P_2-P_1)\diamond M.
$$ 
According to Lemma \ref{xxsec2.4}, we know that $M_{11}=0$. Thus we arrive at
$$
\delta(A_{11}+B_{12}+C_{21})=\delta(A_{11})+\delta(B_{12})+\delta(C_{21}) .
$$
Considering the relations
$$ 
\delta\left(p_n\left ( \frac{1}{2}I,\cdots, \frac{1}{2}I ,P_1,B_{12}+C_{21} +D_{22} \right)\right)
$$ 
and 
$$ 
\delta\left(p_n\left (\frac{1}{2}I,\cdots, \frac{1}{2}I , P_2-P_1,B_{12}+C_{21} +D_{22} \right)\right), 
$$
together with the previous calculations, we assert that 
$$
\delta(B_{12}+C_{21}+D_{22})=\delta(B_{12})+\delta(C_{21})+\delta(D_{22}).
$$

\noindent{\bf Claim 7.} For any $A_{11}\in \mathcal{A}_{11}, B_{12}\in \mathcal{A}_{12},C_{21}\in \mathcal{A}_{21}$ and $D_{22}\in \mathcal{A}_{22}$, we have
$$
 \delta(A_{11}+B_{12}+C_{21}+D_{22})=\delta(A_{11})+\delta(B_{12})+\delta(C_{21})+\delta(D_{22}).
$$ 
 
We only need to prove that
$$
M= \delta(A_{11}+B_{12}+C_{21}+D_{22})-\delta(A_{11})-\delta(B_{12})-\delta(C_{21})-\delta(D_{22})=0.
$$  
Note the facts that 
$$
p_n\left (\frac{1}{2}I,\cdots, \frac{1}{2}I , P_1, D_{22}\right)=0
$$ 
and 
$$p_n\left (\frac{1}{2}I,\cdots, \frac{1}{2}I , P_1, A_{11}+B_{12}+C_{21}+D_{22}\right)=p_n\left (\frac{1}{2}I,\cdots, \frac{1}{2}I , P_1, A_{11}+B_{12}+C_{21}\right).
$$
Applying {\bf Claim 6} (a) yields that
$$
\begin{aligned}
 &p_n \left ( \frac{1}{2}I,\cdots, \frac{1}{2}I , \delta\left(P_1\right), A_{11}+B_{12}+C_{21}+D_{22}\right)\\
 &+ p_n\left (\frac{1}{2}I,\cdots, \frac{1}{2}I , P_1, \delta\left(A_{11}+B_{12}+C_{21}+D_{22}\right) \right) \\
 =&  \delta\left(p_n \left ( \frac{1}{2}I,\cdots, \frac{1}{2}I ,P_1, A_{11}+B_{12}+C_{21}+D_{22}\right)\right) \\
 =&  \delta\left(p_n \left ( \frac{1}{2}I,\cdots, \frac{1}{2}I ,P_1, A_{11}+B_{12}+C_{21}\right)\right)+ \delta\left(p_n \left ( \frac{1}{2}I,\cdots, \frac{1}{2}I ,P_1, D_{22}\right) \right) \\ 
 =&p_n \left ( \frac{1}{2}I,\cdots, \frac{1}{2}I , \delta\left(P_1\right), A_{11}+B_{12}+C_{21}\right)+p_n\left ( \frac{1}{2}I,\cdots, \frac{1}{2}I , \delta\left(P_1\right),D_{22}\right)\\
 &+ p_n\left ( \frac{1}{2}I,\cdots, \frac{1}{2}I , P_1, \delta\left(A_{11}+B_{12}+C_{21}\right) \right) + p_n\left (\frac{1}{2}I,\cdots, \frac{1}{2}I ,  P_1, \delta\left(D_{22}\right) \right) \\ 
=&p_n \left ( \frac{1}{2}I,\cdots, \frac{1}{2}I , \delta\left(P_1\right), A_{11}+B_{12}+C_{21}+D_{22}\right)\\
&+ p_n\left ( \frac{1}{2}I,\cdots, \frac{1}{2}I , P_1, \delta\left(A_{11}\right)+\delta\left(B_{12}\right)+\delta\left(C_{21}\right)+\left(D_{22}\right) \right).
\end{aligned}
$$
Thus we obtain
$$
 0=p_n\left ( \frac{1}{2}I,\cdots, \frac{1}{2}I ,P_1,  M \right)= P_1 \diamond M. 
$$
So $M_{12}=M_{21}=M_{11}=0$ by Lemma \ref{xxsec2.4}.
\vspace{2mm}
 
Similarly, using the relations
$$
p_n\left ( \frac{1}{2}I,\cdots, \frac{1}{2}I ,P_2, A_{11}\right)=0
$$ 
and 
$$
p_n\left (\frac{1}{2}I,\cdots, \frac{1}{2}I , P_2, A_{11}+B_{12}+C_{21}+D_{22}\right)=p_n\left (\frac{1}{2}I,\cdots, \frac{1}{2}I , P_2, B_{12}+C_{21}+D_{22}\right),
$$
one can get $M_{22}=0$.  The proof of this claim is completed.
\vspace{2mm}

\noindent{\bf Claim 8.} For any $A_{ij},B_{ij}\in \mathcal{A}_{ij}\, (i,j=1,2)$, we have
$$
\delta(A_{ij}+B_{ij})=\delta(A_{ij})+\delta(B_{ij}).
$$  
 
\noindent{\bf Case 1:}  $i \neq j$.
 
Note that
$$
\begin{aligned}
 p_n\left (\frac{1}{2}I,\cdots, \frac{1}{2}I , P_i+A_{ij}, P_j+B_{ij}\right) =( P_i+A_{ij}) \diamond (P_j+B_{ij}) \\
 =A_{ij}+B_{ij}+A_{ij}^{\ast}+B_{ij} A_{ij}^{\ast}.
\end{aligned}
$$
In light of {\bf Claim 6}, we know that 
$$
 \begin{aligned}
\ \ \  &\delta\left(p_n\left (\frac{1}{2}I,\cdots, \frac{1}{2}I , P_i+A_{ij}, P_j+B_{ij}\right)\right) \\
 =&\delta(A_{ij}+B_{ij})+\delta\left(A_{ij}^{\ast}\right)+\delta\left( B_{ij} A_{ij}^{\ast}\right).
\end{aligned}\eqno(3.2)
$$ 
 On the other hand, we by {\bf Claim 3} and {\bf Claim 4} have
 $$
 \begin{aligned}
&\delta\left(p_n\left ( \frac{1}{2}I,\cdots, \frac{1}{2}I ,P_i+A_{ij}, P_j+B_{ij}\right)\right) \\
=&p_n\left (\frac{1}{2}I,\cdots, \frac{1}{2}I , \delta\left(P_i+A_{ij}\right), P_j+B_{ij}\right)+ p_n\left ( \frac{1}{2}I,\cdots, \frac{1}{2}I ,P_i+A_{ij}, \delta\left(P_j+B_{ij}\right) \right)\\
=&p_n\left ( \frac{1}{2}I,\cdots, \frac{1}{2}I ,\delta\left(P_i \right), P_j\right)+ p_n\left (\frac{1}{2}I,\cdots, \frac{1}{2}I , \delta\left(A_{ij}\right),P_j\right)\\
 &+p_n\left ( \frac{1}{2}I,\cdots, \frac{1}{2}I ,\delta\left(P_i\right), B_{ij}\right)+ p_n\left (\frac{1}{2}I,\cdots, \frac{1}{2}I ,\delta\left(A_{ij}\right), B_{ij} \right)\\
&+p_n\left (\frac{1}{2}I,\cdots, \frac{1}{2}I , P_i, \delta\left(P_j\right)\right)+ p_n\left ( \frac{1}{2}I,\cdots, \frac{1}{2}I ,P_i, \delta\left(B_{ij}\right) \right)\\
&+p_n\left (\frac{1}{2}I,\cdots, \frac{1}{2}I ,A_{ij},  \delta\left(P_j\right)\right)+ p_n\left ( \frac{1}{2}I,\cdots, \frac{1}{2}I ,A_{ij}, \delta\left(B_{ij}\right) \right)\\
=&\delta\left(p_n\left ( \frac{1}{2}I,\cdots, \frac{1}{2}I ,P_i, P_j\right)\right)+ \delta\left(p_n\left (\frac{1}{2}I,\cdots, \frac{1}{2}I ,P_i, B_{ij}\right) \right)\\
&+\delta\left(p_n\left (\frac{1}{2}I,\cdots, \frac{1}{2}I ,A_{ij},  P_j\right)\right)+ \delta\left(p_n\left (\frac{1}{2}I,\cdots, \frac{1}{2}I ,A_{ij}, B_{ij}\right) \right)\\
=&\delta\left( P_i \diamond B_{ij}\right)+\delta\left(A_{ij} \diamond P_j\right)+ \delta\left(A_{ij}\diamond B_{ij}\right)\\
 =&\delta(B_{ij})+\delta\left(A_{ij}+A_{ij}^{\ast}\right)+\delta\left(B_{ij} A_{ij}^{\ast}\right)\\
 =&\delta\left(B_{ij}\right)+\delta\left(A_{ij}\right)+\delta\left(A_{ij}^{\ast}\right)+\delta\left(B_{ij} A_{ij}^{\ast}\right).
\end{aligned}\eqno(3.3)
$$   
Compare (3.2) with (3.3) gives
$$
\delta\left(B_{ij}+A_{ij}\right)=\delta\left(B_{ij}\right)+\delta\left(A_{ij}\right). 
$$

\noindent{\bf Case 2:}  $i=j$. 
\vspace{2mm}
  
Let us set $M=\delta(A_{ii}+B_{ii})-\delta(A_{ii})-\delta(B_{ii})$. Let us take $l=1,2$, but $l \neq i$. Since
$$
p_n\left (\frac{1}{2}I,\cdots, \frac{1}{2}I , P_l, A_{ii}\right) =0
$$
and
$$
p_n\left (\frac{1}{2}I,\cdots, \frac{1}{2}I , P_l, B_{ii}\right) =p_n\left (\frac{1}{2}I,\cdots, \frac{1}{2}I , P_l, A_{ii}+B_{ii}\right)=0,
$$ 
we know that
 $$
 \begin{aligned}
& p_n\left (\frac{1}{2}I,\cdots, \frac{1}{2}I , \delta\left(P_l\right), A_{ii}+B_{ii}\right)+p_n\left (\frac{1}{2}I,\cdots, \frac{1}{2}I , P_l,\delta\left(A_{ii}+B_{ii}\right) \right) \\
=&\delta\left(p_n\left (\frac{1}{2}I,\cdots, \frac{1}{2}I , P_l, A_{ii}+B_{ii}\right)\right) \\
=&\delta\left(p_n\left (\frac{1}{2}I,\cdots, \frac{1}{2}I , P_l, A_{ii}\right)\right)+\delta\left(p_n\left (\frac{1}{2}I,\cdots, \frac{1}{2}I , P_l, B_{ii}\right)\right) \\
=& p_n\left (\frac{1}{2}I,\cdots, \frac{1}{2}I , \delta\left( P_l\right),A_{ii}\right)+p_n\left (\frac{1}{2}I,\cdots, \frac{1}{2}I ,\delta\left( P_l\right), B_{ii}\right)\\
&+p_n\left ( \frac{1}{2}I,\cdots, \frac{1}{2}I ,P_l, \delta\left(A_{ii}\right) \right)+p_n\left ( \frac{1}{2}I,\cdots, \frac{1}{2}I ,P_l, \delta\left(B_{ii}\right) \right). 
\end{aligned}
$$
Then we have
$$
p_n\left ( \frac{1}{2}I,\cdots, \frac{1}{2}I ,P_l, M \right)=P_l \diamond M=0.
$$ 
By invoking of Lemma \ref{xxsec2.4}, we arrive at $M_{li}=M_{il}=M_{ll}=0$. 
\vspace{2mm}
  
The last step is to show that $M_{ii}=0$. Since 
$$
\begin{aligned}
  p_n\left ( \frac{1}{2}I,\cdots, \frac{1}{2}I ,P_l,C_{li}, A_{ii}\right)=&C_{li}A_{ii}+A_{ii}C_{li}^{\ast},\\
  p_n\left (\frac{1}{2}I,\cdots, \frac{1}{2}I ,P_l,C_{li}, B_{ii}\right)=&C_{li}B_{ii}+B_{ii}C_{li}^{\ast},\\
\end{aligned}
$$ 
and by {\bf Case 1} of this claim and {\bf Claim 4}, we have
$$
\begin{aligned}
&p_n\left (\frac{1}{2}I,\cdots, \frac{1}{2}I , \delta(P_l),C_{li},A_{ii}+B_{ii}\right)+p_n\left (\frac{1}{2}I,\cdots, \frac{1}{2}I ,P_l,\delta(C_{li}), A_{ii}+B_{ii}\right) \\ 
  &+p_n\left ( \frac{1}{2}I,\cdots, \frac{1}{2}I ,P_l,C_{li}, \delta(A_{ii}+B_{ii})\right)  \\
  =&\delta\left( p_n\left( \frac{1}{2}I,\cdots, \frac{1}{2}I ,P_l,C_{li}, A_{ii}+B_{ii}\right) \right)\\ 
=&\delta\left( p_3\left(P_l,C_{li}, A_{ii}+B_{ii}\right) \right)\\ 
  =&\delta \left(  C_{li}(A_{ii}+B_{ii})+(A_{ii}+B{ii})C_{li}^{\ast}\right)\\
  =&\delta (C_{li}A_{ii}+A_{ii}C_{li}^{\ast})+\delta ( C_{li}B_{ii}+B{ii}C_{li}^{\ast})\\  
  =&\delta\left( p_n\left( \frac{1}{2}I,\cdots, \frac{1}{2}I ,P_l,C_{li}, A_{ii}\right) \right)+\delta\left( p_n\left( \frac{1}{2}I,\cdots, \frac{1}{2}I , P_l,C_{li},B_{ii}\right) \right)\\ 
  =&p_n\left (\frac{1}{2}I,\cdots, \frac{1}{2}I ,\delta(P_l),C_{li}, A_{ii}\right)+ p_n\left (\frac{1}{2}I,\cdots, \frac{1}{2}I ,\delta(P_l),C_{li}, B_{ii}\right)\\
  &+p_n\left (\frac{1}{2}I,\cdots, \frac{1}{2}I ,P_l,\delta(C_{li}), A_{ii}\right) +p_n\left (\frac{1}{2}I,\cdots, \frac{1}{2}I ,P_l,\delta(C_{li}), B_{ii}\right)\\ 
&+p_n\left ( \frac{1}{2}I,\cdots, \frac{1}{2}I ,P_l,C_{li}, \delta(A_{ii})\right)  +p_n\left ( \frac{1}{2}I,\cdots, \frac{1}{2}I ,P_l,C_{li}, \delta(B_{ii})\right)  \\  
\end{aligned}
$$
Thus we obtain
$$
\begin{aligned}
0=&p_n\left (\frac{1}{2}I,\cdots, \frac{1}{2}I ,P_l, C_{li},M\right)\\
=& p_{3}\left ( P_l,C_{li},M\right) \\
=&C_{li}\diamond M.
\end{aligned}
$$  
It follows that $C_{li}M_{ii}+M_{ii}C_{li}^{\ast}=0$. That is, $M_{ii}C_{li}^{\ast}=0$ for 
all $C_{li} \in \mathcal{A}_{li}$. Note that $\overline{I-P}=I$. In light of Lemma \ref{xxsec2.2}, 
we conclude that $M_{ii}=0$.
\vspace{2mm}

As an immediate consequence of the previous Claims, we have

\vspace{2mm}

\noindent {\bf Claim 9.}  $\delta$ is an additive mapping.

\vspace{2mm}
 
Let us next show that $\delta$ is a $\ast$-derivation.

\vspace{2mm}
 
\noindent {\bf Claim 10.}  For any $A\in \mathcal{A}$, we have  $\delta(A^{\ast})=\delta(A)^{\ast}$.

\vspace{2mm}
 
In view of Lemma \ref{xxsec2.3}, we konw that
$$
p_n\left ( A, \frac{1}{2}I,\cdots, \frac{1}{2}I \right)=\frac{1}{2}(A+A^{\ast}).
$$
It follows that
$$
\begin{aligned}
\frac{1}{2}(\delta(A)+\delta(A^{\ast}))=&\delta\left(p_n\left ( A,  \frac{1}{2}I,\cdots, \frac{1}{2}I \right)\right)\\
=&p_n\left ( \delta\left(A\right),  \frac{1}{2}I,\cdots, \frac{1}{2}I \right)\\
=&\frac{1}{2}(\delta(A)+\delta(A)^{\ast}).
\end{aligned}
$$
Thus we get $\delta(A^{\ast})=\delta(A)^{\ast}$.

\vspace{2mm} 
 
We next prove that $\delta$ is actually a derivation.

\vspace{2mm}

\noindent {\bf Claim 11.}  For any $A,B \in \mathcal{A}$, we have  $\delta(AB)=\delta(A)B+A\delta(B).$

Since
$$
p_n\left (\frac{1}{2}I,\cdots, \frac{1}{2}I, A, B\right) =A\diamond B=AB+BA^{\ast},
$$
we obtain
$$
\begin{aligned}
\delta( AB+BA^{\ast})=&\delta\left( p_n\left (\frac{1}{2}I,\cdots, \frac{1}{2}I , A,B\right) \right)\\
=&p_n\left (\frac{1}{2}I,\cdots, \frac{1}{2}I , \delta\left( A\right),B\right) +p_n\left (\frac{1}{2}I,\cdots, \frac{1}{2}I , A,\delta\left( B\right) \right)\\
=&\delta(A)\diamond B+A\diamond \delta(B)\\
=&\delta(A)B+B\delta(A)^{\ast}+A\delta(B)+\delta(B)A^{\ast}.
 \end{aligned}
$$ 
It follows that
$$
\delta( AB)+\delta(BA^{\ast})=\delta(A)B+B\delta(A)^{\ast}+A\delta(B)+\delta(B)A^{\ast}.
\eqno(3.4)
$$ 
Replacing $A$ (resp. $B$) by $\text{i}A$ (resp. $\text{i}B$) in (3.4) and using {\bf Claim 9}, we arrive at
$$
\delta( AB)-\delta(BA^{\ast})=\delta(A)B-B\delta(A)^{\ast}+A\delta(B)-\delta(B)A^{\ast}.\eqno(3.5)
$$ 
Combining (3.4) with (3.5) gives
$$
\delta( AB)=\delta(A)B+A\delta(B).
$$ 
\end{proof}

By an analogous manner, we can prove

\begin{theorem}\label{xxsec3.2}
Let $\mathcal{B(H)}$ be the algebra of all bounded linear operators on a complex Hilbert space 
$\mathcal{H}$ and $\mathcal{A}\subseteq \mathcal{B(H)}$ be a factor von Neumann algebra. 
Then a mapping $\delta: \mathcal{A} \longrightarrow  \mathcal{A}$ satisfies the rule
$$
\delta(p_n(A_1, A_2,\cdots, A_n))=\sum_{k=1}^n
p_n(A_1,\cdots, A_{k-1}, \delta(A_k), A_{k+1},\cdots, A_n)
$$
for all $A_1, A_2,\cdots, A_n \in \mathcal{A}$ if and only if $\delta$ is an additive $\ast$-derivation.
\end{theorem}

$\mathcal{B}(\mathcal{H})$ denotes the algebra of all bounded linear 
operators on a complex Hilbert space $\mathcal{H}$. Let us denote the subalgebra of 
all bounded finite rank operators by $\mathcal{F}(\mathcal{H})\subseteq \mathcal{B}(\mathcal{H})$.
We call a subalgebra $\mathcal{A}$ of $\mathcal{B(H)}$ a \textit{standard operator algebra} 
if it contains $\mathcal{F}(\mathcal{H})$. It should be remarked that a standard operator 
algebra is not necessarily closed in the sense of weak operator topology. This is quite 
different from von Neumann algebras which are always weakly closed.

From ring theoretic prespective, standard operator algebras and factor von Neumann 
algebras are both prime, whereas von Neumann algebras are usually semiprime. 
Recall that an algebra $\mathcal{A}$ is \textit{prime} if $A\mathcal{A}B=\{ 0\}$ impliess 
either $A=0$ or $B=0$. An algebra is \textit{semiprime} if $A\mathcal{A}A=\{ 0\}$ impliess $A=0$. 
Every standard operator algebra has the center $\Bbb{C}I$, which is also the center of  arbitrary 
factor von Neumann algebra. An operator $P\in \mathcal{B}(\mathcal{H})$ is said to be 
a \textit{projection} provided $P^{\ast}=P$ and $P^2=P$. Any operator $A \in \mathcal{B}(\mathcal{H})$
can be expressed as $A=\mathfrak{R}A+\text{i}\mathfrak{I}A$, where i is the 
imaginary unit, $\mathfrak{R}A=\frac{A+A^{\ast}}{2}$ and $\mathfrak{I}A=\frac{A-A^{\ast}}{2\text{i}}$. 
Note that both
 $\mathfrak{R}A$ and  $\mathfrak{I}A$ are self-adjoint.

Combining our current methods with the techniques of \cite{Lin1}, one can get

\begin{theorem}\label{xxsec3.3}
Let $\mathcal{H}$ be an infinite dimensional complex Hilbert space and $\mathcal{A}$ be a 
standard operator algebra on $\mathcal{H}$ containing the identity operator $I$.
Suppose that $\mathcal{A}$ is closed under the adjoint operation.  
Then a mapping $\delta: \mathcal{A} \longrightarrow  \mathcal{B(H)}$ satisfies the rule
$$
\delta(p_n(A_1, A_2,\cdots, A_n))=\sum_{k=1}^n
p_n(A_1,\cdots, A_{k-1}, \delta(A_k), A_{k+1},\cdots, A_n)
$$
for all $A_1, A_2,\cdots, A_n \in \mathcal{A}$ if and only if $\delta$ is an additive $\ast$-derivation.
\end{theorem}

We must point out that the technical routes and proving methods of Theorems \ref{xxsec3.2} and \ref{xxsec3.3} 
are fairly similar to those of Theorem \ref{xxsec3.1}, and hence its proofs are omitted here for saving space.

\section{Related Topics for Future Research}
\label{xxsec4}

The main purpose of this article is to concentrate on studying nonlinear $\ast$-Jordan-type derivations 
on operator algebras. The involved operator algebras are based on the algebra $\mathcal{B(H)}$ 
of all bounded linear operators on a complex Hilbert space $\mathcal{H}$, such as standard operator algebras, 
factor von Neumann algebras, von Neumann algebras without central summands of type $I_1$. 
Note that, unlike von Neumann algebras which are always weakly closed, 
a standard operator algebra is not necessarily closed.  The current
work together with \cite{Jing, LiLuFang2, LiZhaoChen, Lin1, Lin2, TaghaviRohiDarvish, Zhang, YuZhang, ZhaoLi1} 
indicates that it is feasible to investigate $\ast$-Jordan-type derivations and $\ast$-Lie-type 
derivations on operator algebras under a unified framework---$\eta$-$\ast$-Jordan-type derivations. 
We have good reasons to believe that characterizing $\eta$-$\ast$-Jordan-type 
derivations on operator algebras is also of great interest. In the light of the motivation 
and contents of this article, we would like to end this article by proposing  several open questions.

Let $\mathcal{A}$ be an associative $\ast$-algebra over the complex field $\mathbb{C}$
and $\eta$ be a non-zero scalar.  For any $A, B\in \mathcal{A}$, we can denote a ``new product" of $A$ and $B$ by 
$A\diamond_{\eta} B=AB+\eta BA^\ast$. This new product $\diamond_{\eta}$ is 
usually said to be \textit{$\eta$-$\ast$-Jordan product}.  Clearly, $1$-$\ast$-Jordan product $\diamond_1$ is the 
so-called $\ast$-Jordan product, and $(-1)$-$\ast$-Jordan product $\diamond_{-1}$ is the so-called $\ast$-Lie product. Therefore, 
it is reasonable to say that $\eta$-$\ast$-Jordan products organically unify  
$\ast$-Jordan products with $\ast$-Lie products.  There are considerable works which are 
devoted to the study of mappings preserving the $\eta$-$\ast$-Jordan product 
between $\ast$-algebras, see \cite{CuiLi, DaiLu, HuoZhengLiu, LiLu, LiLuFang1, Molnar4, Molnar5, ZhaoLi2} 
and the references therein.

Let $\delta: \mathcal{A}\longrightarrow \mathcal{A}$ be a mapping 
(without the additivity assumption). We say that $\delta$ is a 
\textit{nonlinear $\eta$-$\ast$-Jordan derivation} if
$$
\delta(A \diamond_{\eta} B)=\delta(A) \diamond_{\eta} B+A \diamond_{\eta} \delta(B),
$$
holds true for all $A,B\in \mathcal{A}$. Similarly, a mapping $\delta:
\mathcal{A}\longrightarrow \mathcal{A}$ is called a \textit{nonlinear $\eta$-$\ast$-Jordan triple derivation} if it satisfies the condition
$$
\delta(A \diamond_{\eta} B \diamond_{\eta} C)=\delta(A) \diamond_{\eta} B \diamond_{\eta} C+A \diamond_{\eta} \delta(B)\diamond_{\eta} C+A \diamond_{\eta} B \diamond_{\eta} \delta(C)
$$
for all $A, B, C\in \mathcal{A}$, where $A \diamond_{\eta} B \diamond_{\eta} C=(A \diamond_{\eta} B) \diamond_{\eta} C$ . 
We should note that $\diamond$ is not necessarily associative.  

Taking into account the definitions of $\eta$-$\ast$-Jordan derivations and $\eta$-$\ast$-Jordan 
triple derivations, one can propose one much more common notion. Suppose that $n\geq 2$ is a fixed positive
integer. Let us see a sequence of polynomials with scalar $\eta$ and $\ast$
$$
\begin{aligned}
p_1(x_1)&=x_1,\\
p_2(x_1,x_2)&=x_1 \diamond_{\eta} x_2=x_1x_2+\eta x_2x_1^\ast,\\
p_3(x_1,x_2,x_3)&=p_2(x_1,x_2) \diamond_{\eta} x_3=(x_1 \diamond_{\eta} x_2) \diamond_{\eta} x_3,\\
p_4(x_1,x_2,x_3,x_4)&=p_3(x_1,x_2,x_3) \diamond_{\eta} x_4=((x_1 \diamond_{\eta} x_2) \diamond_{\eta} x_3) \diamond_{\eta} x_4,\\
\cdots &\cdots,\\
p_n(x_1,x_2,\cdots,x_n)&=p_{n-1}(x_1,x_2,\cdots,x_{n-1}) \diamond_{\eta} x_n\\
&=\underbrace{(\cdots ((}_{n-2}x_1 \diamond_{\eta} x_2)\diamond_{\eta} x_3)\diamond_{\eta} \cdots \diamond_{\eta} x_{n-1}) \diamond_{\eta} x_n.
\end{aligned}
$$
Accordingly, a \textit{nonlinear $\eta$-$\ast$-Jordan $n$-derivation} is a mapping $\delta:
\mathcal{A} \longrightarrow \mathcal{A}$ satisfying the condition
$$
\delta(p_n(A_1, A_2,\cdots, A_n))=\sum_{k=1}^n
p_n(A_1,\cdots, A_{k-1}, \delta(A_k), A_{k+1},\cdots, A_n)
$$
for all $A_1,A_2,\cdots, A_n\in \mathcal{A}$. This notion is motivated by the definition of $\ast$-Jordan-type derivations 
and that of $\ast$-Lie-type derivations. Then each $\ast$-Jordan derivation is a 
$1$-$\ast$-Jordan 2-derivation and every $\ast$-Jordan triple derivation is a $1$-$\ast$-Jordan 3-derivation. 
Likewise, each $\ast$-Lie derivation is a 
$(-1)$-$\ast$-Jordan 2-derivation and every $\ast$-Lie triple derivation is a $(-1)$-$\ast$-Jordan 3-derivation. 
$\eta$-$\ast$-Jordan 2-derivations, $\eta$-$\ast$-Jordan 3-derivations and $\eta$-$\ast$-Jordan $n$-derivations 
are collectively referred to as \textit{$\eta$-$\ast$-Jordan-type derivations}. 
$\eta$-$\ast$-Jordan-type derivations on operator algebras are intensively studied by 
several authors, \cite{Jing, LiLuFang2, LiZhaoChen, Lin1, Lin2, TaghaviRohiDarvish, Zhang, YuZhang, ZhaoLi1} . 
A basic question in this line is to investigate whether each nonlinear $\eta$-$\ast$-Jordan-type derivation 
on an operator algebra $\mathcal{A}$ with $\ast$ is an additive $\ast$-derivation. 
In view of the current work and existing results in this direction, we propose several open questions.

\begin{question}\label{xxsec4.1}
Let $\mathcal{H}$ be an infinite dimensional complex Hilbert space and $\mathcal{A}$ be a 
standard operator algebra on $\mathcal{H}$ containing the identity operator $I$. 
Let $\eta$ be a non-zero scalar. 
Suppose that $\mathcal{A}$ is closed under the adjoint operation.  
A mapping $\delta: \mathcal{A} \longrightarrow  \mathcal{B(H)}$ satisfies the following condition:
$$
\delta(p_n(A_1, A_2,\cdots, A_n))=\sum_{k=1}^n
p_n(A_1,\cdots, A_{k-1}, \delta(A_k), A_{k+1},\cdots, A_n)
$$
for all $A_1, A_2,\cdots, A_n \in \mathcal{A}$. Is $\delta$ an additive $\ast$-derivation ? Does 
the relation $\delta(\eta A)=\eta \delta(A)$ hold for any $A\in \mathcal{A}$ ?
\end{question}

\begin{question}\label{xxsec4.2}
Let $\mathcal{B(H)}$ be the algebra of all bounded linear operators on a complex Hilbert space 
$\mathcal{H}$ and $\mathcal{A}\subseteq \mathcal{B(H)}$ be a factor von Neumann algebra. Suppose that $\eta$ is a non-zero scaler. 
Let $\delta: \mathcal{A} \longrightarrow  \mathcal{B(H)}$ be a mapping such that
$$
\delta(p_n(A_1, A_2,\cdots, A_n))=\sum_{k=1}^n
p_n(A_1,\cdots, A_{k-1}, \delta(A_k), A_{k+1},\cdots, A_n)
$$
for all $A_1, A_2,\cdots, A_n \in \mathcal{A}$. Is $\delta$ an additive $\ast$-derivation ? Do we have
the relation $\delta(\eta A)=\eta \delta(A)$ for any $A\in \mathcal{A}$ ?
\end{question}

\begin{question}\label{xxsec4.3}
Let $\mathcal{B(H)}$ be the algebra of all bounded linear operators on a complex Hilbert space 
$\mathcal{H}$ and $\mathcal{A}\subseteq \mathcal{B(H)}$ be a von Neumann algebra without central 
summands of type $I_1$. Let $\eta$ be a non-zereo scalar. A mapping $\delta: \mathcal{A}\longrightarrow \mathcal{B(H)}$ 
satisfies the following conditions:
$$
\delta(p_n(A_1, A_2,\cdots, A_n))=\sum_{k=1}^n
p_n(A_1,\cdots, A_{k-1}, \delta(A_k), A_{k+1},\cdots, A_n)
$$
for all $A_1, A_2,\cdots, A_n \in \mathcal{A}$. Is $\delta$ an additive $\ast$-derivation ? Can we get 
the relation $\delta(\eta A)=\eta \delta(A)$ for any $A\in \mathcal{A}$ ?
\end{question}



\end{document}